\documentclass[12pt,reqno]{amsart}
\usepackage{amssymb,a4wide,eucal,amscd,yhmath}
\usepackage{fancyhdr}
\usepackage{amsmath}
\usepackage{mathrsfs}
\usepackage[all,cmtip]{xy}
\SilentMatrices
\SelectTips{eu}{}
\newtheorem{theorem}{Theorem}[section]
\newtheorem{proposition}[theorem]{Proposition}

\newtheorem{lemma}[theorem]{Lemma}

\theoremstyle{rem}

\theoremstyle{definition}
\newtheorem{definition}[theorem]{Definition}
\theoremstyle{construct}

\theoremstyle{examp}
\newtheorem{examp}[theorem]{Example}
\newcommand\projective\mathbf
\newcommand\PP{\projective P}

\newcommand\OO{\mathcal O}

\newcommand\ZZ{\mathbb Z}

\newcommand\onto\twoheadrightarrow
\newcommand\lra\longrightarrow
\newcommand\dar\downarrow
\DeclareMathOperator{\pic}{Pic}
\DeclareMathOperator{\im}{im}
\DeclareMathOperator{\cok}{coker}
\DeclareMathOperator{\rk}{rank}

\begin{document}

\title{Vector Bundles associated to Monads on Cartesian Products of Projective Spaces}
\author{Damian M Maingi}
\date{July, 2021}

\address{Department of Mathematics\\College of Science\\
\\Sultan Qaboos University\\ P.O Box 50, 123 Muscat, Sultanate of Oman}
\email{dmaingi@squ.edu.om}

\maketitle

\begin{abstract}
In this paper we construct vector bundles associated to monads  on $X=\PP^n\times\PP^n\times\PP^m\times\PP^m$.
We first establish the existence of such monads on $X$. Once the monads exist, the next natural question is whether the 
cohomology vector bundle associated to these monads are simple or not.
We study these vector bundles associated to monads and prove their stability and simplicity.
\end{abstract}

\section{Introduction}
\noindent In Algebraic Geometry, one very interesting problem deals with the existence of indecomposable low rank
vector bundles on algebraic varieties. One very popular technique to construct vector bundles is via monads which 
appear in many contexts within algebraic geometry. Monads were first introduced by Horrocks,\cite{4} who showed  
that all vector bundles $E$ on $\PP^3$ could be obtained as the cohomology bundle of a monad of the following kind:
\[
\xymatrix
{
0\ar[r] & \oplus_i\OO_{\PP^3}(a_i) \ar[r]^{A} & \oplus_j\OO_{\PP^3}(b_j) \ar[r]^{B} & \oplus_n\OO_{\PP^3}(c_n) \ar[r] & 0
}
\]
where $A$ and $B$ are matrices whose entries are homogeneous polynomials of degrees $b_j-a_i$ and $c_n-b_j$ respectively for some integers $i,j,n$.\\
\\
The first problem is to tackle the existence of monads. Fl\o{}ystad in \cite{2} gave sufficient and necessary conditions for the existence of monads over the projective space. 
Costa and Miro-Roig in \cite{1} extended these results to smooth quadric hypersurfaces of dimension at least 3.
Marchesi, Marques and Soares in \cite{7} generalized Fl\o{}ystad's theorem to a larger set of varieties.\\
\\
In this paper we first establish the existense of monads
\[\begin{CD}0@>>>{\OO_X(-1,-1,-1,-1)^{\oplus k}} @>>^{f}>{\mathscr{G}_n\oplus
\mathscr{G}_m}@>>^{g}>\OO_X(1,1,1,1)^{\oplus k} @>>>0\end{CD}\]
on $X=\PP^n\times\PP^n\times\PP^m\times\PP^m$ where
$\mathscr{G}_n:=\OO_X(0,-1,0,0)^{\oplus n+k}\oplus\OO_X(-1,0,0,0)^{\oplus n+k}$ and $\mathscr{G}_m:=\OO_X(0,0,-1,0)^{\oplus m+k}\oplus\OO_X(0,0,0,-1)^{\oplus m+k}$.
We then prove stability of the kernel bundle $\ker g$ and finally prove that the cohomology vector bundle, $\ker g/\im f$ is simple.

\section{Preliminaries}

\begin{definition}
Let $X$ be a non-singular irreducible projective variety of dimension $d$ and let $\mathscr{L}$ be an ample line bundle on $X$. For a 
torsion-free sheaf $F$ on $X$ we define
\begin{enumerate}
 \item the degree of $F$ relative to $\mathscr{L}$ as $\deg_{\mathscr{L}}F:= c_1(F)\cdot \mathscr{L}^{d-1}$, where $c_1(F)$ is the first chern class of $F$
 \item the slope of $F$ as $\mu_{\mathscr{L}}(F):= \frac{c_1(F)\mathscr{L}^{d-1}}{rk(F)}$ and
 \end{enumerate}
\end{definition}

\begin{definition}
Let $X$ be an algebraic variety and let $E$ be a torsion-free sheaf on $X$. Then $E$ is $\mathscr{L}$-stable
if every subsheaf $F\hookrightarrow E$ satisfies $\mu_{\mathscr{L}}(F)<\mu_{\mathscr{L}}(E)$, where $\mathscr{L}$ is an ample invertible sheaf.
\end{definition}

\subsection{Hoppe's Criterion over cyclic varieties.}
Suppose that the picard group Pic$(X) \simeq \ZZ$ such varieties are called cyclic.
Given a locally free sheaf (or, equivalently, a holomorphic vector bundle) $E\rightarrow X$, 
there is a unique integer $k_E$ such that $-r + 1 \leq c_1(E(-k_E)) \leq 0$.
Setting $E_{norm} := E(-k_E)$, we say $E$ is normalized if $E = E_{norm}$. 
Then one has the following stability criterion:

\begin{proposition} [\cite{3}, Lemma 2.6]
Let $E$ be a rank $r$ holomorphic vector bundle over a cyclic projective variety $X$. 
If $H^0((\bigwedge^q E)_{norm}) = 0$  for  $1\leq q\leq r-1$, then E is stable.
If $H^0((\bigwedge^q E)_{norm}(-1)) = 0$ for $1\leq q \leq r-1$, then E is semistable.
\end{proposition}

\subsection{Hoppe's Criterion over polycyclic varieties.}
Suppose that the picard group Pic$(X) \simeq \ZZ^l$ where $l\geq2$ is an integer then $X$ is a polycyclic variety.
Given a divisor $B$ on $X$ we define $\delta_{\mathscr{L}}(B):= deg_{\mathscr{L}}\OO_{X}(B)$.
Then one has the following stability criterion {\cite{5}, Theorem 3}:

\begin{theorem}[Generalized Hoppe Criterion]
 Let $G\rightarrow X$ be a holomorphic vector bundle of rank $r\geq2$ over a polycyclic variety $X$ equiped with a polarisation 
 $\mathscr{L}$ if
 \[H^0(X,(\wedge^sG)\otimes\OO_X(B))=0\] 
 for all $B\in\pic(X)$ and $s\in\{1,\ldots,r-1\}$ such that
 $\begin{CD}\displaystyle{\delta_{\mathscr{L}}(B)<-s\mu_{\mathscr{L}}(G)}\end{CD}$ then $G$ is stable and if
 $\begin{CD}\displaystyle{\delta_{\mathscr{L}}(B)\leq-s\mu_{\mathscr{L}}(G)}\end{CD}$ then $G$ is semi-stable.\\
 Conversely if then $G$ is (semi-)stable then  \[H^0(X,G\otimes\OO_X(B))=0\]
 for all $B\in\pic(X)$ and all $s\in\{1,\ldots,r-1\}$ such that
 $\delta_{\mathscr{L}}(B)<-s\mu_{\mathscr{L}}(G)$ or $\delta_{\mathscr{L}}(B)\leq-s\mu_{\mathscr{L}}(G)$.
\end{theorem}

\subsection{Hoppe's Criterion over $X=\PP^n\times\PP^n\times\PP^m\times\PP^m$}
Suppose the ambient space is $X=\PP^n\times\PP^n\times\PP^m\times\PP^m$ then $\pic(X) \simeq \ZZ^4$.
We denote by $\langle a, b,c,d\rangle$ the generators of $\pic(X)$.
Denote by $\OO_X(a,b,c,d):= {p_1}^*\OO_{\PP^n}(a)\otimes {p_2}^*\OO_{\PP^m}(b)\otimes {p_3}^*\OO_{\PP^m}(c)\otimes {p_4}^*\OO_{\PP^m}(d)$, where $p_1$ and $p_2$ are natural projections  from $X$ 
to $\PP^n$ and $p_3$ and $p_4$ are natural projections  from $X$ to $\PP^m$.
For any line bundle $\mathscr{L} = \OO_X(a,b,c,d)$ on $X$ and a vector bundle $E$, we will write $E(a,b,c,d) = E\otimes\OO_X(a,b,c,d)$ 
and $(a,b,c,d):= a[h\times\PP^n]+b[\PP^n\times t]+c[h\times\PP^m]+d[\PP^m\times t]$ to represent its corresponding divisor.
The normalization of $E$ on $X$ with respect to $\mathscr{L}$ is defined as follows:
Set $d=\deg_{\mathscr{L}}(\OO_X(1,0,0,0))$, since $\deg_{\mathscr{L}}(E(-k_E,0,0,0))=\deg_{\mathscr{L}}(E)-4k\cdot \rk(E)$ 
there's a unique integer $k_E:=\lceil\mu_\mathscr{L}(E)/d\rceil$ such that $1 - d.\rk(E)\leq \deg_\mathscr{L}(E(-k_E,0,0,0))\leq0$. 
The twisted bundle $E_{{\mathscr{L}}-norm}:= E(-k_E,0,0,0)$ is called the $\mathscr{L}$-normalization of $E$.
Finally we define the linear functional $\delta_{\mathscr{L}}$ on  $\mathbb{Z}^4$ as $\delta_{\mathscr{L}}(p_1,p_2,p_3,p_4):= \deg_{\mathscr{L}}\OO_{X}(p_1,p_2,p_3,p_4)$.

\begin{proposition}[\cite{6}, Proposition 6]
Let $X$ be a polycyclic variety with Picard number $2$, let $\mathscr{L}$ be an ample line bundle and
let E be a rank $r>1 $ holomorphic vector bundle over $X$.
If $H^0((\bigwedge^q E)_{{\mathscr{L}}-norm}(p_1,p_2)) = 0$ for $1\leq q \leq r-1$ and every $(p_1,p_2)\in \mathbb{Z}^2$ such that $\delta_{\mathscr{L}}\leq0$
then E is $\mathscr{L}$-stable.
\end{proposition}

\begin{definition}
A vector bundle $E$ is said to be
\begin{enumerate}
 \item decomposable if it is isomorphic to a direct sum $E_1\oplus E_2$ of two non-zero
vector bundles, otherwise $E$ is indecomposable.
\item simple if its only endomorphisms are the homotheties i.e. Hom$(E,E)=k$ which is equivalent
to $h^0(X,E\otimes E^*)=1$.
\end{enumerate}
\end{definition}

\begin{proposition}
Let $0\rightarrow E \rightarrow F \rightarrow G\rightarrow0$ be an exact sequence of vector bundles. \\
Then we have the following exact sequences involving exterior and symmetric powers:\\
\begin{enumerate}
 \item $0\lra\bigwedge^q E \lra\bigwedge^q F \lra\bigwedge^{q-1} F\otimes G\lra\cdots \lra F\otimes S^{q-1}G \lra S^{q}G\lra0$\\
 \item $0\lra S^{q}E \lra S^{q-1}E\otimes F \lra\cdots \lra E\otimes\bigwedge^{q-1}F\lra\bigwedge^q F \lra\bigwedge^q G\lra 0$
\end{enumerate}
\end{proposition}

\begin{theorem}[K\"{u}nneth formula]
 Let $X$ and $Y$ be projective varieties over a field $k$. 
 Let $\mathscr{F}$ and $\mathscr{G}$ be coherent sheaves on $X$ and $Y$ respectively.
 Let $\mathscr{F}\boxtimes\mathscr{G}$ denote $p_1^*(\mathscr{F})\otimes p_2^*(\mathscr{G})$\\
 then $\displaystyle{H^m(X\times Y,\mathscr{F}\boxtimes\mathscr{G}) \cong \bigoplus_{p+q=m} H^p(X,\mathscr{F})\otimes H^q(Y,\mathscr{G})}$.
\end{theorem}

\noindent Since for our case we deal $X = \PP^n\times\PP^n\times\PP^m\times\PP^m$ then \\
$\displaystyle{H^t(\OO_X (i,j,k,l)) \cong \bigoplus_{p+q+r+s=t} H^p(\PP^n,\OO_{\PP^n}(i))\otimes H^q(\PP^n,\OO_{\PP^n}(j))\otimes H^r(\PP^m,\OO_{\PP^m}(k)))\otimes H^s(\PP^m,\OO_{\PP^m}(l))}$ 
where $p,q,r,s,t,i,j,k$ and $l$ are integers.

\begin{theorem}[\cite{9}, Theorem 4.1]
 Let $n\geq1$ be an integer  and $d$ be an integer. We denote by $S_d$ the space of homogeneous polynomials of degree in 
 $n+1$ (conventionally if $d<0$ then $S_d=0$):
 \begin{enumerate}
  \item We have $H^0(\PP^n,\OO_{\PP^n}(d))=S_d$ for all $d$.
  \item We have $H^i(\PP^n,\OO_{\PP^n}(d))=0$ for $1<i<n$ and for all $d$.
  \item The space $H^n(\PP^n,\OO_{\PP^n}(d))\cong H^0(\PP^n,\OO_{\PP^n}(-d-n-1))$.
 \end{enumerate}
\end{theorem}

We adopt a lemma by Jardim and Earp [\cite{6}, Lemma 9] for our purpose in this work.
\begin{lemma}
If $p_1+p_2+p_3+p_4>0$ then $h^p(X,\OO_X (-p_1,-p_2,-p_3,-p_4)^{\oplus k}) = 0$ where $X = \PP^n\times\PP^n\times\PP^m\times\PP^m$ and for $0\leq p< \dim(X) -1$, for $k$ a non negative integer.
\end{lemma}

\begin{lemma}[\cite{6}, Lemma 10]
Let $A$ and $B$ be vector bundles canonically pulled back from $A'$ on $\PP^n$ and $B'$ on $\PP^m$ then\\
$\displaystyle{H^q(\bigwedge^s(A\otimes B))=
\sum_{k_1+\cdots+k_s=q}\big\{\bigoplus_{i=1}^{s}(\sum_{j=0}^s\sum_{m=0}^{k_i}H^m(\wedge^j(A))\otimes(H^{k_i-m}(\wedge^{s-j}(B)))) \big\}}$.
\end{lemma}
\noindent The lemma above depends on the following facts:

\begin{center}
$\displaystyle{H^q(A_1\oplus\cdots\oplus A_s) = \sum_{k_1+\cdots+k_s=q}\big\{\bigoplus_{i=1}^{s}H^k_i(A_i)\big\}}$.\\
$\displaystyle{H^q(A\otimes B)=\sum_{m=0}^qH^m(A)\otimes H^{q-m}(B)}$.\\
$\displaystyle{\wedge^s(A\otimes B)=\sum_{j=0}^s\wedge^j(A)\otimes\wedge^{s-j}(B)}$.
\end{center}

\subsection{Background on Monads}

\begin{definition}
Let $X$ be a nonsingular projective variety. A {\it{monad}} on $X$ is a complex of vector bundles:
\[ M_\bullet:
\xymatrix{0\ar[r] & M_0 \ar[r]^{\alpha} & M_1 \ar[r]^{\beta} & M_2 \ar[r] & 0}
\]
with $\alpha$ injective and $\beta$ surjective equivalently, $M_\bullet$ is a monad if $\alpha$ and $\beta$ are of maximal rank and $\beta\circ\alpha = 0$.
\end{definition}


\begin{definition}
A monad as defined above has a display diagram of short exact sequences as shown below:\\
\[
\begin{CD}
@.@.0@.0\\
@.@.@VVV@VVV\\
0@>>>{M_0} @>>>\ker{\beta}@>>>E@>>>0\\
@.||@.@VVV@VVV\\
0@>>>{M_0} @>>^{\alpha}>{M_1}@>>>\cok{\alpha}@>>>0\\
@.@.@V^{\beta}VV@VVV\\
@.@.{M_2}@={M_2}\\
@.@.@VVV@VVV\\
@.@.0@.0
\end{CD}
\]
The kernel of the map $\beta$, $\ker\beta$ and the cokernel of $\alpha$, $\cok\alpha$ for the given monad are also vector bundles and the vector bundle
$E = \ker(\beta)/\im (\alpha)$ and is called the cohomology bundle of the monad.
\end{definition}

\begin{definition}\cite{8}
Let $X$ be a nonsingular projective variety, let $\mathscr{L}$ be a very ample line sheaf, and $V,W,U$ be finite dimensional $k$-vector spaces.
A linear monad on $X$ is a complex of sheaves,
\[ 
\xymatrix
{
0\ar[r] & V\otimes {\mathscr{L}}^{-1} \ar[r]^{\alpha} & W\otimes \OO_X \ar[r]^{\beta} & U\otimes \mathscr{L} \ar[r] & 0
}
\]
where $\alpha\in $Hom$(V,W)\otimes H^0 \mathscr{L}$ is injective and $\beta\in $Hom$(W,U)\otimes H^0 \mathscr{L}$ is surjective.
\end{definition}

\begin{definition}\cite{8}
A torsion-free sheaf $E$ on $X$ is said to be a {\it{linear sheaf}}  on $X$ if it can be represented as the
cohomology sheaf of a linear monad i.e.
$E= ker(\beta)/im (\alpha)$, moreover $\rk(E) = w - u - v$, where $w=\dim W$, $v=\dim V$ and $u=\dim U$.
\end{definition}

\section{Main Results}
\noindent The goal of this section is to construct monads over the cartesian product $X=\PP^n\times\PP^n\times\PP^m\times\PP^m$ of projective spaces.
We then proceed  to prove stability and simplicity of the cohomology bundle associated to such monads. 
We start by recalling the existence and classification of linear monads on $\PP^n$ given by
Fl\o{}ystad in \cite{4}.

\begin{lemma}[\cite{2}, Main Theorem] Let $k\geq1$. There exists monads on $\PP^k$ whose maps are matrices of linear forms,
\[
\begin{CD}
0@>>>{\OO_{\PP^{k}}(-1)^{\oplus a}} @>>^{A}>{\OO^{\oplus b}_{\PP^{k}}} @>>^{B}>{\OO_{\PP^{k}}(1)^{\oplus c}} @>>>0\\
\end{CD}
\]
if and only if at least one of the following is fulfilled;\\
$(1)b\geq2c+k-1$ , $b\geq a+c$ and \\
$(2)b\geq a+c+k$
\end{lemma}

\section{Main Results}

\begin{lemma}
Let $n,m$ and $k$ are positive integers, given four matrices $f_1,f_2,f_3$ and $f_4$ of order $k$ by $n+k$, and 
four other matrices $g_1,g_2,g_3$ and $g_4$ of order $n+k$ by $k$ as shown;
\[ f_1 =\left[ \begin{array}{ccccc}
&y_n \cdots y_0 \\
\adots&\adots \\
y_n \cdots y_0 \end{array} \right]_{k\times {(n+k)}}\]

\[ f_2 =\left[ \begin{array}{ccccccc}
&x_n \cdots  x_0\\
\adots&\adots \\
x_n \cdots  x_0 \end{array} \right]_{k\times {(n+k)}}\]

\[ f_3 =\left[ \begin{array}{ccccccc}
&t_m \cdots  t_0\\
\adots&\adots \\
t_m \cdots  t_0 \end{array} \right]_{k\times {(m+k)}}\]

\[ f_4 =\left[ \begin{array}{ccccccc}
&z_m \cdots  z_0\\
\adots&\adots \\
z_m \cdots z_0 \end{array} \right]_{k\times {(m+k)}}\]

\[ g_1 =\left[\begin{array}{cccccc}
x_0\\
\vdots &\ddots
 & x_0\\
x_n &\ddots &\vdots\\
&& x_n
         \end{array} \right]_{{(n+k)}\times k}\] 

\[ g_2 =\left[\begin{array}{cccccc}
y_0\\
\vdots &\ddots
 & y_0\\
y_n &\ddots &\vdots\\
&& y_n
         \end{array} \right]_{{(n+k)}\times k}\]

\[ g_3 =\left[ \begin{array}{cccccc}
z_0\\
\vdots &\ddots
 & z_0\\
z_m &\ddots &\vdots\\
&& z_m
         \end{array} \right]_{{(m+k)}\times k}\] and

\[ g_4 =\left[\begin{array}{cccccc}
t_0\\
\vdots &\ddots
 & t_0\\
t_m &\ddots &\vdots\\
&& t_m
         \end{array} \right]_{{(m+k)}\times k}\]

we define two matrices $f$ and $g$ as follows\\
\[ f =\left[\begin{array}{cccc}
f_1 & -f_2 & f_3 & -f_4
         \end{array} \right]\] and

\[ g =\left[\begin{array}{cc}
g_1 \\ g_2 \\g_3 \\ g_4
         \end{array} \right].\]
Then we have:\\
(i) $f\cdot g = 0$ and
\\
(ii) The matrices $f$ and $g$ have maximal rank
\end{lemma}

\begin{proof}
(i) Since $f_1\cdot g_1 = f_2\cdot g_2$, $f_3\cdot g_3 = f_4\cdot g_4$ then we have that
\[ f\cdot g =\left[ \begin{array}{cccc} f_1 & -f_2 & f_3 & -f_4 \end{array} \right]
 \left[ \begin{array}{cc} g_1 \\ g_2 \\ g_3 \\ g_4\end{array} \right]\] 
\[ =\left[\begin{array}{cccc} f_1g_1 & -f_2g_2 & f_3g_3 & -f_4g_4 
         \end{array} \right]=0\]
(ii) Notice that the rank of the two matrices drops if and only if all $x_0,...,x_n$, $y_0,...,y_n$, $z_0,...,z_m$
and $t_0,...,t_m$ are zeros. Hence maximal rank.
\end{proof}

\noindent Using the matrices given in the above lemma we are going to construct a monad.

\begin{theorem}
Let $n,m$ and $k$ be positive integers. Then there exists a linear monad on $X = \PP^n\times\PP^n\times\PP^m\times\PP^m$ of the form;
\[
\begin{CD}
0@>>>{\OO_X(-1,-1,-1,-1)^{\oplus k}} @>>^{f}>{\mathscr{G}_n\oplus
\mathscr{G}_m}@>>^{g}>\OO_X(1,1,1,1)^{\oplus k} @>>>0
\end{CD}
\]
where $\mathscr{G}_n:=\OO_X(0,-1,0,0)^{\oplus n+k}\oplus\OO_X(-1,0,0,0)^{\oplus n+k}$ and \\
\hspace*{1cm} $\mathscr{G}_m:=\OO_X(0,0,-1,0)^{\oplus m+k}\oplus\OO_X(0,0,0,-1)^{\oplus m+k}$.

\end{theorem}

\begin{proof}
The maps $f$ and $g$ in the monad are the matrices given in Lemma 3.4.\\
Notice that\\
$f\in$ Hom$(\OO_X(-1,-1,-1,-1)^{\oplus k},\mathscr{G}_n\oplus\mathscr{G}_m)$ and \\
$g\in$ Hom$(\mathscr{G}_n\oplus\mathscr{G}_m,\OO_X(1,1,1,1)^{\oplus k})$. \\
Hence by the above lemma they define the desired monad.
\end{proof}

\begin{theorem}
Let $T$ be a vector bundle on $X=\PP^n\times\PP^n\times\PP^m\times\PP^m$ defined by the sequence
\[\begin{CD}0@>>>T @>>>\mathscr{G}_n\oplus\mathscr{G}_m@>>^{g}>\OO_X(1,1,1,1)^{\oplus k} @>>>0\end{CD}\]
where $\mathscr{G}_n:=\OO_X(0,-1,0,0)^{\oplus n+k}\oplus\OO_X(-1,0,0,0)^{\oplus n+k}$ and \\
\hspace*{1cm} $\mathscr{G}_m:=\OO_X(0,0,-1,0)^{\oplus m+k}\oplus\OO_X(0,0,0,-1)^{\oplus m+k}$
then $T$ is stable for an ample line bundle $\mathscr{L} = \OO_X(1,1,1,1)$
\end{theorem}

\begin{proof}
We need to show that $H^0(X,\bigwedge^q T(-p_1,-p_2,-p_3,-p_4))=0$ for all $p_1+p_2+p_3+p_4>0$ and $1\leq q\leq \rk(T)$.\\
\\
Consider the ample line bundle $\mathscr{L} = \OO_X(1,1,1,1) = \OO(L)$. \\
Its class in 
$\pic(X)= \langle [a\times\PP^n],[\PP^n\times b],[c\times\PP^n],[\PP^n\times d]\rangle$ corresponds to the class\\
$1\cdot[a\times\PP^n]+1\cdot[\PP^n\times b]\cdot[c\times\PP^m]+1\cdot[\PP^m\times d]$, where $a$ and $b$ are hyperplanes of $\PP^n$ and 
$c$ and $d$ hyperplanes of $\PP^m$ with the intersection product induced by $a^{n} = b^{n} = c^{m} = d^{m}=1$ 
and $a^{n+1} = b^{n+1} = c^{n+1} = d^{n+1=0}$ 
\\
Now from the display diagram of the monad we get \\ 
\begin{align*}
\begin{split}
c_1(T) & = c_1(\mathscr{G}_n\oplus\mathscr{G}_m) - c_1(\OO_X(1,1,1,1)^{\oplus k})\\
       & = (n+k)(-1,0,0,0)+(n+k)(0,-1,0,0)+(m+k)(0,0,-1,0)+(m+k)(0,0,0,-1) - k(1,1) \\
       & = (-n-2k,-n-2k,-m-2k,-m-2k) \\
\end{split}
\end{align*}
Since $L^{2n+2m}>0$ the degree of $T$ is $\deg_{\mathscr{L}}T = c_1(T)\cdot\mathscr{L}^{d-1}$\\
\begin{align*}
\begin{split}
=-(n+m+4k)([a\times\PP^n]+[\PP^n\times b]+[c\times\PP^m]+[\PP^m\times d])\cdot \\
(1\cdot[a\times\PP^n]+1\cdot[\PP^n\times b]+1\cdot[c\times\PP^m]+1\cdot[\PP^m\times d])^{2n+2m-1}\\
\end{split}
\end{align*}
\begin{center}$=-(n+m+4k)L^{2n+2m}< 0$\end{center}
Since $\deg_{\mathscr{L}}T<0$, then $(\bigwedge^q T)_{\mathscr{L}-norm} = (\bigwedge^q T)$ and  it suffices by 
the generalized Hoppe Criterion (Proposition 2.5), to prove that $h^0(\bigwedge^q T(-p_1,-p_2,-p_3,-p_4)) = 0$
with $p_1+p_2+p_3+p_4\geq0$ and for all $1\leq q\leq rk(T)-1$.\\
\\
Next we twist the exact sequence 
\[\begin{CD}0@>>>T @>>>\mathscr{G}_n\oplus\mathscr{G}_m@>>^{g}>\OO_X(1,1,1,1)^{\oplus k} @>>>0\end{CD}\]
by $\OO_X(-p_1,-p_2,-p_3,-p_4)$ we get,
\[
0\lra T(-p_1,-p_2,-p_3,-p_4)\lra\mathscr{\overline{G}}_n\oplus\mathscr{\overline{G}}_m\lra\OO_X(1-p_1,1-p_2,1-p_3,1-p_4)^{\oplus k}\lra0\]
where $\mathscr{\overline{G}}_n:=\OO_X(-1-p_1,-p_2,-p_3,-p_4)^{\oplus n+k}\oplus\OO_X(-p_1,-1-p_2,-p_3,-p_4)^{\oplus n+k}$ and\\
$\mathscr{\overline{G}}_m:=\OO_X(-p_1,-p_2,-1-p_3,-p_4)^{\oplus m+k}\oplus\OO_X(-p_1,-p_2,-p_3,-1-p_4)^{\oplus m+k}$
and taking the exterior powers of the sequence by Proposition 2.7 we get
\[0\lra \bigwedge^q T(-p_1,-p_2,-p_3,-p_4) \lra \bigwedge^q \mathscr{\overline{G}}_n\oplus\mathscr{\overline{G}}_m\lra \cdots\].
\\
Taking cohomology we have the injection:
\[0\lra H^0(X,\bigwedge^{q}T(-p_1,-p_2,-p_3,-p_4))\hookrightarrow H^0(X,\bigwedge^q \mathscr{\overline{G}}_n\oplus\mathscr{\overline{G}}_m)\].
From Lemma 2.10 and 2.11 $H^0(X,\bigwedge^q \mathscr{\overline{G}}_n\oplus\mathscr{\overline{G}}_m)=0$.\\
\\
$\Longrightarrow$ $h^0(X,\bigwedge^{q}T(-p_1,-p_2,-p_3,-p_4)) =  h^0(X,\bigwedge^q \mathscr{\overline{G}}_n\oplus\mathscr{\overline{G}}_m)=0$\\
\\
i.e. $h^0(\bigwedge^{q}T(-p_1,-p_2,-p_3,-p_4))=0$ and thus $T$ is stable.

\end{proof}

\begin{theorem} Let $X=\PP^n\times\PP^n\times\PP^m\times\PP^m$, then the cohomology vector bundle $E$ associated to the monad 
\[
\begin{CD}
0@>>>{\OO_X(-1,-1,-1,-1)^{\oplus k}} @>>^{f}>{\mathscr{G}_n\oplus\mathscr{G}_m}@>>^{g}>\OO_X(1,1,1,1)^{\oplus k} @>>>0
\end{CD}
\]
of rank $2n+2m+2k$ is simple.
\end{theorem}

\begin{proof}
The display of the monad is
\[
\begin{CD}
@.@.0@.0\\
@.@.@VVV@VVV\\
0@>>>{\OO_{X}(-1,-1,-1,-1)^{\oplus k}} @>>>T@>>>E@>>>0\\
@.||@.@VVV@VVV\\
0@>>>{\OO_{X}(-1,-1,-1,-1)^{\oplus k}} @>>^{f}>{\mathscr{G}_n\oplus\mathscr{G}_m}@>>>Q@>>>0\\
@.@.@V^{g}VV@VVV\\
@.@.{\OO_{X}(1,1,1,1)^{\oplus k}}@={\OO_{X}(1,1,1,1)^{\oplus k} }\\
@.@.@VVV@VVV\\
@.@.0@.0
\end{CD}
\]

\noindent Since $T$ the kernel of the map $g$ is stable from the above theorem 4.3, we prove that the cohomology vector bundle $E=\ker g/\im f$  is simple.\\
\\
The first step is to take the dual of the short exact sequence 
\[\begin{CD}
0@>>>\OO_X(-1,-1,-1,1)^{\oplus k} @>>>T@>>>E @>>>0
\end{CD}\]
to get
\[
\begin{CD}
0@>>>E^* @>>>T^* @>>>\OO_X(1,1,1,1)^{\oplus k}@>>>0.
\end{CD}
\]
Tensoring by $E$ we get\\
\[
\begin{CD}
0@>>>E\otimes E^* @>>>E\otimes T^* @>>>E(1,1,1,1)^k@>>>0.
\end{CD}
\]
\\
Now taking cohomology gives:
\[\begin{CD}
0@>>>H^0(X,E\otimes E^*) @>>>H^0(X,E\otimes T^*) @>>>H^0(E(1,1,1,1)^{\oplus k})@>>>\cdots
\end{CD}\]
\\
which implies that 
\begin{equation}
h^0(X,E\otimes E^*) \leq h^0(X,E\otimes T^*)
\end{equation}
\\
Now we dualize the short exact sequence
\[\begin{CD}
0@>>>T @>>>{\mathscr{G}_n\oplus\mathscr{G}_m} @>>>\OO_X(1,1,1,1)^{\oplus k} @>>>0
\end{CD}\]
\\
to get
\[\begin{CD}
0@>>>\OO_X(-1,-1,-1,-1)^{\oplus k} @>>>{\mathscr{G}_n\oplus\mathscr{G}_m} @>>>T^* @>>>0
\end{CD}\]
\\
For the sake of brevity we shall use the notation $H^q(\OO_X^a)$ in place of $H^q(X,\OO_X^{\oplus a})$.
\\
Now twisting by $\OO_X(-1,-1,-1,-1)$ and taking cohomology and get\\
\[\begin{CD}
0\lra H^0(\OO_X(-2,-2,-2,-2)^k) \lra H^0(\mathscr{\overline{G}}_n\oplus\mathscr{\overline{G}}_m)\lra H^0(T^*(-1,-1,-1,-1))\lra\\
\lra H^1(\OO_X(-2,-2,-2,-2)^k) \lra H^1(\mathscr{\overline{G}}_n\oplus\mathscr{\overline{G}}_m)\lra H^1(T^*(-1,-1,-1,-1))\lra\\
\lra H^2(X,\OO_X(-2,-2,-2,-2)^k) \lra H^2(\mathscr{\overline{G}}_n\oplus\mathscr{\overline{G}}_m)\lra H^2(T^*(-1,-1,-1,-1))\lra\cdots
\end{CD}
\]
\\
from which we deduce $H^0(X,T^*(-1,-1,-1,-1)) = 0$ and $H^1(X,T^*(-1,-1,-1,-1)) = 0$ from Theorems 2.8 and 2.9.\\
\\
Lastly, tensor the short exact sequence
\[
\begin{CD}
0@>>>\OO(-1,-1,-1,-1)^{\oplus k} @>>>T @>>> E@>>>0\\
\end{CD}
\]
by $T^*$ to get
\[
\begin{CD}
0@>>>T^*(-1,-1,-1,-1)^k @>>>T\otimes T^* @>>> E\otimes T^*@>>>0\\
\end{CD}
\]
and taking cohomology we have
\\
\[
\begin{CD}
0@>>>H^0(X,T^*(-1,-1,-1,-1)^k) @>>>H^0(X,T\otimes T^*) @>>> H^0(X,E\otimes T^*)@>>>\\
@>>>H^1(X,T^*(-1,-1,-1,-1)^k)@>>>\cdots
\end{CD}
\]
\\
But $H^1(X,T^*(-1,-,-1,-11)^k=0$ for $k>1$ from above.\\
\\
so we have 
\\
\[
\begin{CD}
0@>>>H^0(X,T^*(-1,-1,-1,-1)^{k}) @>>>H^0(X,T\otimes T^*) @>>> H^0(X,E\otimes T^*)@>>>0
\end{CD}
\]
\\
This implies that 
\begin{equation}
h^0(X,T\otimes T^*) \leq h^0(X,E\otimes T^*)
\end{equation}
\\
Since $T$ is stable then it follows that it is simple which implies $h^0(X,T\otimes T^*)=1$.\\
\\
From $(1)$ and now $(2)$ and putting these together we have;\\
\[1\leq h^0(X,E\otimes E^*) \leq h^0(X,E\otimes T^*) = h^0(X,T\otimes T^*) = 1\]\\
\\
We have $ h^0(X,E\otimes E^*) = 1 $ and therefore $E$ is simple.

\end{proof}

\begin{examp}
We construct a monad on $X = \PP^1\times\PP^1\times\PP^2\times\PP^2$ by explicitly giving the maps 
$f$ and $g$.
We define $f$ and $g$ as follows:
\[ f :=\left( \begin{array}{cccc|cccc|ccccc|ccccc}
0   & 0   & y_1 & y_0 & 0    & 0    & -x_1 & -x_0 & 0   & 0   & t_2 & t_1 & t_0 & 0 & 0 & -z_2 & -z_1 & -z_0\\
0   & y_1 & y_0 & 0   & 0    & -x_1 & -x_0 &    0 & t_2 & t_1 & t_0 & 0   & 0   & 0    & -z_2  & -z_1 & -z_0 & 0\\
y_1 & y_0 & 0   & 0   & -x_1 & -x_0 & 0    &    0 & t_2 & t_1 & t_0 & 0   & 0   & -z_2 & -z_1  & -z_0 &    0 & 0\\
\end{array} \right)
\]
and
\[ g :=\left( \begin{array}{ccc}
x_0 & 0   & 0 \\
x_1 & x_0 & 0  \\
  0 & x_1 & x_0 \\
  0 &   0 & x_1 \\
\hline
y_0 & 0   & 0 \\
y_1 & y_0 & 0  \\
  0 & y_1 & y_0 \\
  0 &   0 & y_1 \\
\hline
z_0 & 0   & 0 \\
z_1 & z_0 & 0  \\
z_2 & z_1 & z_0 \\
  0 & z_2 & z_1 \\
  0 & z_2 & z_1 \\
  0 &   0 & z_2 \\
\hline
t_0 & 0   & 0 \\
t_1 & t_0 & 0  \\
t_2 & t_1 & t_0 \\
  0 & t_2 & t_1 \\
  0 & t_2 & t_1 \\
  0 &   0 & t_2 \\
\end{array} \right)
\]
from $f$ and $g$ we have the monad
\[
\begin{CD}
0@>>>{\OO_X(-1,-1,-1,-1)^{\oplus3}} @>>^{f}>{\mathscr{G}_n\oplus
\mathscr{G}_m}@>>^{g}>\OO_X(1,1,1,1)^{\oplus3} @>>>0
\end{CD}
\]
where $\mathscr{G}_n:=\OO_X(0,-1,0,0)^{\oplus4}\oplus\OO_X(-1,0,0,0)^{\oplus4}$ and \\
\hspace*{1cm} $\mathscr{G}_m:=\OO_X(0,0,-1,0)^{\oplus5}\oplus\OO_X(0,0,0,-1)^{\oplus5}$.

\end{examp}

\end{document}